\documentclass{amsart}
\usepackage{amsmath, amssymb, graphicx, amsthm}
\usepackage[bookmarks=true]{hyperref}

\newtheorem{theorem}{Theorem}
\newtheorem{lemma}{Lemma}
\newtheorem{corollary}{Corollary}
\newtheorem{definition}{Definition}
\newtheorem*{maintheorem}{Theorem 3.1}
\newtheorem*{maintheoremB}{Theorem 3.2}
\newtheorem*{maintheoremC}{Theorem 4.1}
\newtheorem*{maintheoremD}{Theorem 4.2}

\newtheorem*{caolitheorem}{Theorem (Cao-Li \cite{CaoLi2011})}

\newcommand{\eterm}{e^{-|x|^2/4}}
\newcommand{\intS}{\int_\Sigma \lambda^2}
\newcommand{\diverg}{{\text{div}}}

\newcommand{\la}{\langle}
\newcommand{\ra}{\rangle}
\newcommand{\nablaS}{\nabla ^ \Sigma}

\newcommand{\tr}{\text{Tr}}

\newcommand{\ltwo}{L^2_{\lambda^2}(\Sigma)}
\newcommand{\na}{{N_\alpha}}
\newcommand{\nb}{{N_\beta}}

\numberwithin{theorem}{section}
\numberwithin{lemma}{section}
\numberwithin{corollary}{section}
\numberwithin{equation}{section}

\begin{document}
\title{Gaussian Harmonic Forms and Two-Dimensional Self-Shrinkers}
\author {Matthew McGonagle}
\address{Department of Mathematics, Johns Hopkins University, 3400 North Charles Street, Baltimore, MD 21218-2686, USA}
\email{mcgonagle@math.jhu.edu}
\date{3/29/2012}
\maketitle

\begin{abstract}

We consider 2-dimensional orientable self-shrinkers $\Sigma$ for the Mean Curvature Flow of polynomial volume growth immersed in $\mathbb R^n$.  We look at closed one forms minimizing the norm $\int_\Sigma \eterm |\omega|^2$ in their cohomology class. Any closed form satisfying the Euler-Lagrange equation for this minimization will be called a Gaussian Harmonic one Form (GHF). 

We then use these forms to show that if such a $\Sigma$ has genus $\geq 1,$ then we have a lower bound on the supremum norm of $A^2$. GHF's may also be applied to create an upperbound for the lowest eigenvalue of the operator $L$. In the codimension one case $\Sigma \to \mathbb R^3$, for certain conditions on the principal curvatures we use GHF's to get a lower bound on the index of $L$ depending on the genus $g$. Likewise, in the compact  codimension one case we obtain an estimate of the lowest eigenvalue of $L$ and also on $\inf |x|^2$.
\end{abstract}


\setcounter{section}{0}
\section*{Introduction} 

A surface immeresd in $\mathbb R^n$ is considered a self-shrinker for the Mean Curvature Flow if it satisfies
$$ H = \frac{ x^N}{2}$$
where $x$ is the position vector in $\mathbb R^n$. Self-Shrinkers are important to the Mean Curvature Flow, because they are related to the singularities of the flow \cite{CM2009}.

For a normal vector field $N$ and tangent vector fields $X,Y$ we define the Second Fundamental Form by $A^N(X,Y) = \la \nabla_X^E N,Y \ra$. We also define $B(X,Y) = - \nabla^N_X Y$ and $H = B(i,i)$. Our sign convention is chosen to be consistent with that of Colding-Minicozzi \cite{CM2009}. 

For this paper, we will use $\Sigma$ to be a complete two dimensional orientable self-shrinker immersed in $\mathbb R^n$ of polynomial volume growth. We will use $K$ to refer to the curvature of $\Sigma$. For an orthonormal frame $\na$ in the normal bundle $N\Sigma$, Gauss' Equation gives us that $K = A^\na _{11}A^\na_{22}- A^\na_{12}A^\na_{12}$.

For each $\na$ we may diaganolize $A^\na$ with eigenvalues $\kappa_{\alpha i}$ for a frame $\{i_\alpha\}$ in $T\Sigma$. Note that from Gauss' Equation we have that $K = \frac{1}{2}(A^\na_{ii}A^\na_{jj} - A^\na_{ij}A^\na_{ij})$ and so $K = \sum\limits_\alpha \kappa_{\alpha 1}\kappa_{\alpha 2}$. We denote the components of a tangential vectorfield $X$ with respect to $i_\alpha$ as $X_{\alpha i}$.

For a self-shrinker $\Sigma$, we expect bounds on the curvature to affect the geometry and topology of $\Sigma.$ An argument, pointed out to the author by Professor Minicozzi, shows that if $|A|^2 \leq c_2 |x|^2$ with $c_2 < 1/8$ then $\Sigma$ has finite topology. For any two-dimensional $\Sigma$ we have $|H|^2 \leq 2 |A|^2$. Since we are on a self-shrinker we then have that $|x^N|^2 \leq 8c_2|x|^2$. We then have that $|x^T|^2 \geq (1-8c_2)|x|^2 > 0$ for $|x|$ sufficiently large. Therefore, we may apply Morse Theory to find that the topology is constant for large enough $|x|$.

In Ros \cite{Ros2006} and  Urbano \cite{Urbano2011} harmonic one forms are used to study the index of a minimal surface of genus $g$. The harmonic forms are constructed by minimizing the $L^2$ norm $\int_\Sigma |\omega|^2$ in a cohomology class. From classical riemann surface theory \cite{FK1980} p. 42, we can associate $2g$ linearly independent $L^2$ harmonic one forms to a surface of genus $g$. Ros \cite{Ros2006} and Urbano \cite{Urbano2011} consider the duals to these harmonic one forms as vectors in Euclidean space. They study the Jacobi operator $\triangle_\Sigma + |A|^2$ acting on the coordinate functions of these vectors.

We consider a parallel situation for a self-shrinker $\Sigma$. We examine closed one forms $\omega$ minimizing $\int_\Sigma \eterm |\omega|^2$ in their cohomology class. Any closed one form satsifying the Euler-Lagrange Equation for this minimization will be called a Gaussian Harmonic one Form (GHF). Note that this norm does not come from a conformal change of $\Sigma$. The measure $\eterm dV$ is related to the variational characterization of self-shrinkers in Colding-Minicozzi \cite{CM2009}. A difference between this and the harmonic case is that we only get $g$ linearly independent GHF, because the Euler-Lagrange condition for GHF is not preserved by a rotation. For simplicity in notation we will define $\lambda^2 \equiv \eterm$. We will also denote $\ltwo$ to be the $L^2$ space associated with the measure $\lambda^2 dV$

We will have need to work with vectors in $T \Sigma$, $N \Sigma$, and in $\mathbb R^n$. We will use the indices $\{i,j,...\}$ for vectors or forms considered intrinsically part of $\Sigma$, $\{a, b, ...\}$ when considering vectors in $\mathbb R^n$, and $\{\alpha, \beta,...\}$ when considering vectors in $N\Sigma$. We will often denote an element of a frame by its index where obvious. 

The vectors $\partial_a$ will denote a standard basis of orthonormal vectors for $\mathbb R^n$. We will make use of the exterior derivative of the coordinate function $x^a$ along $\Sigma$ and will denote this by $dx^a$. That is, $dx^a$ is not the exterior derivative in $\mathbb R^n$. Also, for any vector $V \in \mathbb R^n$ let $V^T$ be the projection onto the tangent space of $\Sigma$. For any one form $\omega$ on $\Sigma$ we shall denote the vector dual to $\omega$ considered to be sitting in $\mathbb R^n$ as $W$.

As in Colding-Minicozzi \cite{CM2009}, we define the operators $\mathcal L$ and $L$ on scalar functions of $\Sigma$ as
$$ \mathcal L u = \triangle u - \frac{1}{2} \nabla_{x^T} u$$
and
$$ L u = \mathcal L u + \frac{1}{2}u + |A|^2 u.$$
Note that $\mathcal L = -\nabla^*\nabla$ for the measure $\eterm dV$. That is, $\intS f \mathcal L g = -\intS \la \nabla f, \nabla g \ra $ for any $f,g \in C^\infty_0(\Sigma)$. In particular, $\mathcal L$ is a symmetric operator. The operator $L$ is associated with the second variation related to the variational characterization of self-shrinkers in Colding-Minicozzi \cite{CM2009}.

Since we will be working with vectors in $T\Sigma$, $N\Sigma$, and $\mathbb R^n$ we will use $\nabla^E$ to denote the Euclidean connection, $\nabla^\Sigma = (\nabla^E)^T$ for the connection on $\Sigma$, and $\nabla^N = (\nabla^E)^N$ for the normal connection.

We also define the operators $\mathcal L^E$ and $L^E$ on $\mathbb R^n$ valued vectorfields V along $\Sigma$ by
$$\mathcal L^E V = \tr_\Sigma \nabla^{2,E} V - \frac{1}{2}\nabla^E_{ x^T} V$$
$$ = \nabla^{2,E}_{i,i} V -  \frac{1}{2}\nabla^E_{x^T} V $$
and
$$ L^E V = \mathcal L^E V + \frac{1}{2} V + |A|^2 V.$$
Note that for $V = V^a \partial_a$ we have $\mathcal L^E V = (\mathcal L V^a)\partial_a$ and $L^E V = (L V^a) \partial_a$.

We will likewise use the $\nabla^\Sigma$ connection to define $\mathcal L^\Sigma$ and $L^\Sigma$ on tensors intrinsic to $\Sigma$.

For the first section of the paper we establish some computations for GHF's on $\Sigma$. The computation that we will be using in later sections is contained in Lemma 1.1. For a GHF on a self-shrinker we have that
$$\mathcal L^\Sigma \omega (v) = \frac{1}{2}\omega(v)  -  A^\na(W,j)A^\na(j,v).  $$

For the second section of the paper, we compute $\mathcal L^E W$ for a GHF $\omega$ on a self-shrinker. From Lemma 2.2, we have that
$$ \mathcal L^E W = -2\la \nabla^\Sigma \omega, A^\nb \ra \nb +\frac{1}{2}W - 2 A^\nb (W, i) A^\nb(i,j)j.$$

For the third section, we apply our computations for GHF's to prove two results in the general co-dimension case. The first is a type of ``gap theorem'' for the genus of a self-shrinker.

\begin{maintheorem}
If $\Sigma$ is a 2-dimensional orientable self-shrinker of polynomial volume growth immersed in $\mathbb R^n$ with genus $\geq 1$, then 
$$\sup \limits_{x \in \Sigma, |v|=1} A^\nb(v, i) A^\nb(i,v) \geq 1/2 .$$
\end{maintheorem}

\noindent {\bf Remark:} In the case of $\Sigma \to \mathbb R^3$ we have that $\sup \limits_{x \in \Sigma, |v|=1} A^\nb(v, i) A^\nb(i,v) = \sup \limits_{x\in \Sigma, i}\kappa_i^2$ where the $\kappa_i$ are the principal curvatures of $\Sigma$.

Our result should be compared with the following gap theorem of Cao-Li \cite{CaoLi2011} (after renormalizing to meet our definition of self-shrinker).

\begin{caolitheorem}
If $M^n \to \mathbb R^{n+p}$ ($p \geq 1$) is an n-dimensional complete self-shrinker without boundary and with polynomial volume growth, and satisfies
$$ |A|^2 \leq 1/2$$
then M is one of the following: \\
\indent i) a round sphere in $\mathbb R^{n+1}$ \\
\indent ii) a cylinder in $\mathbb R^{n+1}$ \\
\indent iii) a hyperplane in $\mathbb R^{n+1}.$
\end{caolitheorem}

Cao-Li gives us a gap theorem for $|A|^2$ in any co-dimension. For the co-dimension one case, our result is a gap theorem for the genus based on the size of the largest principal curvature squared $\kappa_i^2$. So, we see that our classification is not completely covered by Cao-Li's result.

Our second result in the general co-dimension case is a bound for the lowest eigenvalue $\eta_0$ of the operator $L$ acting on scalar functions on $\Sigma$.

\begin{maintheoremB}
Let $\Sigma$ be a 2-dimensional orientable self-shrinker of polynomial volume growth immersed in $\mathbb R^n$ with genus $\geq 1.$ The lowest eigenvalue $\eta_0$ of $L$ acting on scalar functions of $\Sigma$ has upper bound given by
$$ \eta_0 \leq -1 +\sup \limits_{x \in \Sigma, |v|=1} A^\nb(v, i) A^\nb(i,v).$$
\end{maintheoremB}

\noindent {\bf Remark:} From Theorem 3.1 we see that the best upper bound Theorem 3.2 can give us is $\eta_0 \leq -1/2$. In the codimension one case $\Sigma \to \mathbb R^3$, Colding-Minicozzi \cite{CM2009} were able to show that $\eta_0 \leq -1$ by showing that $LH=H$ and that $H$ is in the appropriate weighted space. So, we see that the estimate of Theorem 3.2 isn't optimal for the codimension one case.

In the fourth section, we make applications to the codimension one case $\Sigma \to \mathbb R^3$. We first show a lower bound on the index of $L$ based on the genus of $\Sigma$.

\begin{maintheoremC}
Let $\Sigma$ be a 2-dimensional orientable self-shrinker of polynomial volume growth immersed in $\mathbb R^3$ with genus $g$ and principal curvatures $\kappa_i$. If \\
 $|\kappa_1^2 - \kappa_2^2| \leq \delta < 1$, then the index of $L$ acting on scalar functions of $\Sigma$ has a lower bound given by

$$ \text{Index}_\Sigma(L) \geq \frac{g}{3}.$$
\end{maintheoremC}

We then show some estimates for the lowest eigenvalue of $L$ and also $\inf \limits_{x \in \Sigma} |x|^2$ in the compact case of $\Sigma \to \mathbb R^3$.

\begin{maintheoremD}
Let $\Sigma$ be a 2-dimensional orientable compact self-shrinker immersed in $\mathbb R^3$ with genus $g \geq 1$ and principal curvatures $\kappa_i$. Let $\eta_0$ be the lowest eigenvalue of $L$ acting on scalar functions. We have
$$ \eta_0 \leq -3/2 + \sup  \limits_{x\in\Sigma} |\kappa_1^2 - \kappa_2^2| .$$

If $ |\kappa_1^2 - \kappa_2^2| \leq \delta < 5/2,$ then
$$ \inf \limits_{x\in\Sigma} |x|^2 \leq \frac{4}{5/2 - \delta}.$$

\end{maintheoremD}

\noindent {\bf Remark:} As before, we note that Colding-Minicozzi \cite{CM2009} have shown that in the codimension one case that $\eta_0 \leq -1$. Therefore, we see that the estimate for $\eta_0$ of Theorem 4.2 is only optimal for $|\kappa_1^2 - \kappa_2^2| < 1/2.$


\section*{Acknowledgements}
The author would like to thank Professor William Minicozzi and Professor Joel Spruck for their guidance and support.


\section{Gaussian Harmonic One Forms}
Let $\omega_0$ be a closed $C^\infty_0$ one form, i.e. $d \omega_0 = 0$. The classical harmonic form cohomologous to $\omega_0$ is constructed by minimizing the $L^2(\Sigma)$ norm $\int_\Sigma |\omega|^2$ in the cohomology class of $\omega_0$ \cite{Jost2008}. A form $\omega$ is defined to be harmonic if and only if it is closed ($\nablaS \omega$ is symmetric) and co-closed ($\nablaS \omega$ is traceless) \cite{Ros2006}.

Instead of using harmonic one forms to represent the
cohomology class of $\omega_0$  as in Ros \cite{Ros2006} and Urbano \cite{Urbano2011}, we will use the form minimizing the $\ltwo$ norm

$$ \int_\Sigma \lambda^2 |\omega|^2$$

where $\omega = \omega_0 + df$ (Remember, $\lambda^2 \equiv \eterm$). To find the Euler-Lagrange equation, assume the minimum is achieved by a closed form $\omega$. Also, let $\delta$ be the dual to $d$ with respect to the regular euclidean surface measure $dV$. As in Jost\cite{Jost2008}, the minimization gives us that $0 = \delta(\lambda^2 \omega) = - \diverg(\lambda^2 \omega)$. This is equivalent to $\tr_\Sigma \nablaS \omega = - \omega (\nablaS \log \lambda^2) $.

Using that $\lambda^2 = \eterm$, we get

$$ \tr_\Sigma \nablaS \omega = \frac{1}{2} \omega(x^T). $$

Therefore, we make a definition:
\begin{definition}
A form $\omega$ will be called a Gaussian Harmonic Form (GHF) if and only if $\omega$ is closed $(\nablaS \omega$ is symmetric$)$ and $\omega$ is gaussian co-closed $( \tr_\Sigma \nablaS \omega = \frac{1}{2} \omega(x^T))$.
\end{definition}
We have the following result:


\begin{lemma}
Let $\omega$ be a GHF on any surface $\Sigma$ with vector dual $W \in T\Sigma$, then
$$\mathcal L^\Sigma \omega (v) = K \omega(v) + \frac{1}{2}\omega(v)  - \frac{\la x, \na\ra}{2} A^\na(W,v). $$
On a self-shrinker we have
$$\mathcal L^\Sigma \omega (v) = \frac{1}{2}\omega(v)  -  A^\na(W,j)A^\na(j,v).  $$
\end{lemma}

\begin{proof}
We have a Weitzenbock formula $\triangle^\Sigma \omega = - (\delta d + d \delta) \omega + K \omega$ where $\delta d + d \delta $ is the Hodge Laplacian \cite{Jost2008}.
Now $\delta \omega = - \diverg \omega = -\frac{1}{2} \omega (x^T)$ and $d\omega = 0$. Therefore $- (\delta d + d \delta) \omega = \frac{1}{2} d (\omega(x^T)) = \frac{1}{2} \nabla^\Sigma (\omega(x^T))$. Then using a Leibniz rule we get

$$ \triangle^\Sigma \omega (v) = K \omega + \frac{1}{2}\omega(\nablaS_v x^T) + \frac{1}{2} \nablaS \omega (x^T, v).$$

Now, we use that $\nablaS_v x^T = v - \la x, \na \ra A^\na(v,j)j$ and that $\mathcal L^\Sigma = \triangle^\Sigma - \frac{1}{2}\nablaS_{x^T}$ to get the first equation of the lemma.

If $\Sigma$ is a self-shrinker, then we have that
$$ \mathcal L^\Sigma \omega(v) = K \omega(v) + \frac{1}{2}\omega(v) - A^H (v, W).$$

Decomposing in a normal basis $\{\na\}$ we have that
$$K W - \sum_i A^H(W,i)i = \sum_\alpha \sum_i (\kappa_{\alpha 1}\kappa_{\alpha 2} W_{\alpha i} i_\alpha - 
(\kappa_{\alpha 1}+\kappa_{\alpha 2})\kappa_{\alpha i}W_{\alpha i} i_\alpha)$$
$$ = \sum_\alpha \sum_i -\kappa_{\alpha i}^2 W_{\alpha i} i_{\alpha}$$
$$ = - A^\na(W,i)A^\na(i,j)j.$$

\end{proof}

\begin{corollary}
Let $\omega$ be a GHF on any surface $\Sigma$. Let $W$ be the vectorfield dual to $\omega$. We have
$$\mathcal L^\Sigma W = \la \mathcal L^\Sigma \omega, dx^a \ra \partial_a = KW + \frac{1}{2}W - \frac{\la x, \nb \ra}{2} A^ \nb(W,j)j.$$

On a self-shrinker $\Sigma$, we have
\begin{equation}
\mathcal L^\Sigma  W = \la \mathcal L^\Sigma \omega, dx^a \ra \partial_a = \frac{1}{2}W - A^\nb(W,i)A^\nb(i,j)j. 
\end{equation}
\end{corollary}


\section{\texorpdfstring{$\mathcal L^E W$ for a GHF $\omega$}{Euclidean L(W) for a GHF w}}

First, we make some computations for $\nablaS dx^a$ and $\triangle^\Sigma dx^a$. 

\begin{lemma}
For a general surface $\Sigma$, if $W$ is the vectorfield dual to a one form $\omega$, then
$$ \la \nablaS \omega, \nablaS dx^a \ra = -\nb^a \la \nablaS \omega, A^\nb \ra $$ 
and
$$\la \omega, \mathcal L^\Sigma dx^a \ra \partial_a =-A^\nb(W,i)A^\nb(i,j)j - \nabla^N_W H + \frac{\nb}{2}A^\nb(x^T, W).$$
\end{lemma}

\begin{proof}
We note that $\nablaS \partial_a^T = - \la \partial_a, \nb \ra \nablaS \nb$. Hence $\nablaS dx^a = -\nb^a A^\nb $, and so we get
$$\la \nablaS \omega, \nablaS dx^a \ra = -\nb^a \la \nablaS \omega, A^\nb \ra.$$

Fix a point $p \in \Sigma$. Use a tangential frame $\{ j \}$ and a normal frame $\{ \nb \}$ such that $\nablaS j (p) = 0$ and $\nabla^N \nb (p) = 0$. We have that
$$\triangle dx^a (k) = \nabla_j \nabla dx^a (j,k)$$
$$ = -A^\nb( \partial_a^T, j)A^\nb(j,k) - \nb^a \nabla A^\nb(j,j,k).$$
Using the Codazzi Equation, we get
$$\triangle^\Sigma dx^a (k) =  -A^\nb(\partial_a^T, j)A^\nb(j,k) -  (\nabla^N_k H)^a.$$
So
$$\mathcal L^\Sigma dx^a(k) =  -A^\nb(\partial_a^T, j)A^\nb(j,k) -  (\nabla^N_k H)^a + \frac{1}{2} \nb^a A^\nb(x^T, k).$$

\end{proof}

\begin{corollary}
For a self-shrinker $\Sigma$,
$$\la \omega, \mathcal L^\Sigma dx^a \ra\partial_a =-A^\nb(W, i)A^\nb(i,j)j . $$
\end{corollary}
\begin{proof}
Fix a point $p \in \Sigma$. Use a tangential frame $\{ j \}$ and normal frame $\{ \nb \}$ such that $\nablaS j (p) = 0$ and $\nabla^N \nb (p) = 0$. For a self-shrinker, we compute that
$$-(\nabla^N_W H)^a = -\frac{\nb^a}{2}A^\nb ( x^T, W) .$$
\end{proof}


\begin{lemma}
Let $\omega$ be a GHF on any surface $\Sigma$. Let $W$ be the vectorfield dual to $\omega$. We have

$$ \mathcal L^E W = -2\nb \la \nabla^\Sigma \omega, A^\nb \ra + \frac{\nb}{2} A^\nb(x^T, W) - \nabla^N_W H $$
$$+ KW + \frac{1}{2}W - A^\nb(W, i)A^\nb(i,j)j - \frac{\la x, \nb \ra}{2} A^\nb (W,j)j.$$

On a self-shrinker, we have
\begin{equation}
 \mathcal L^E W = -2\la \nabla^\Sigma \omega, A^\nb \ra \nb +\frac{1}{2}W - 2 A^\nb (W, i) A^\nb(i,j)j.
\end{equation}

\end{lemma}

\begin{proof}
Note that $W = \langle \omega, dx^a \rangle \partial_a$. Hence, $\mathcal L^E W = (\mathcal L \langle \omega, dx^a \ra) \partial_a$.

So we have 
$$\mathcal L^E W = \la\mathcal L^\Sigma \omega, dx^a \ra \partial_a + 
2\la\nablaS \omega, \nablaS dx^a \ra \partial_a + 
\la \omega, \mathcal L^\Sigma dx^a \ra \partial_a.$$

From Corollary 1.1 and Lemma 2.1, we know that
$$\la \mathcal L^\Sigma \omega, dx^a \ra \partial_a = KW + \frac{1}{2}W - \frac{\la x, \nb \ra}{2} A^ \nb(W,j)j,$$
$$ 2\la\nablaS \omega, \nablaS dx^\alpha \ra \partial_\alpha  = -2 \nb \la \nablaS \omega, A^\nb \ra,$$
and
$$\la \omega, \mathcal L^\Sigma dx^a \ra \partial_a =-A^\nb(W,i)A^\nb(i,j)j - \nabla^N_W H + \frac{\nb}{2}A^\nb(x^T, W).$$

Putting these three equations together we complete the proof of the first equation. The second equation, for self-shrinkers, is similar.

\end{proof}


\section{Applications to General Co-Dimension}

\begin{theorem}
If $\Sigma$ is a 2-dimensional orientable self-shrinker of polynomial volume growth immersed in $\mathbb R^n$ with genus $\geq 1$, then 
$$\sup \limits_{x \in \Sigma, |v|=1} A^\nb(v, i) A^\nb(i,v) \geq 1/2 .$$
\end{theorem}

\begin{proof}

Assume $\Sigma$ has genus $g\geq 1$. Parallel to the result for classical Riemann Surfaces \cite{FK1980}, we then have $g$ linearly independent GHF in $\ltwo.$ Let $\omega$ be one of these GHF and $W$ be the dual to $\omega$.

Consider any $\phi \in C^\infty_0 (\Sigma)$. We have

$$ 0 \leq \intS |\nabla^\Sigma (\phi W)|^2 =  - \intS \la \phi W, \mathcal L^\Sigma (\phi W) \ra$$
\begin{equation} 
= -\intS |W|^2 \phi \mathcal L \phi - \frac{1}{2}\intS \la \nablaS \phi^2, \nablaS |W|^2\ra - \intS \phi^2 \la W, \mathcal L^\Sigma W \ra. \end{equation}

Now, using integration by parts we have that

$$-\frac{1}{2} \intS \la \nablaS \phi^2, \nablaS |W|^2 \ra = \intS |W|^2 \phi \mathcal L \phi + \intS |W|^2 |\nablaS \phi|^2. $$
Putting this into (3.1), we get that
$$ 0 \leq \intS |W|^2 |\nabla \phi|^2 - \intS \phi^2 \la W, \mathcal L^\Sigma W \ra.$$
Then, using equation (1.1) of Corollary 1.1 we have that
$$ 0 \leq \intS |W|^2 |\nabla \phi|^2+ \intS \phi^2 A^\nb(W,i)A^\nb(i,W) - \frac{1}{2} \intS \phi^2 |W|^2 .$$

Let $M = \sup \limits_{x \in \Sigma, |v|=1} A^\nb(v, i) A^\nb(i,v)$. Using standard cut-off functions of increasing domain and $|\nabla^\Sigma \phi|^2 \leq 1$, we get

$$ 0 \leq (M-\frac{1}{2}) \intS |W|^2 .$$

Since $W \not \equiv 0,$ we get the theorem.

\end{proof}


For any non-compact manifold $\Sigma$ the operator $L$ on scalar functions may not have a nice spectrum, but we may still define the lowest eigenvalue of L by

$$ \eta_0 = \inf \limits_{\phi \in C^\infty_0(\Sigma)} \frac{\intS (|\nabla \phi|^2 - |A|^2 \phi^2 - \frac{1}{2}\phi^2)}{\intS \phi^2}$$

We get an upper bound for $\eta_0$.

\begin{theorem}
Let $\Sigma$ be a 2-dimensional orientable self-shrinker of polynomial volume growth immersed in $\mathbb R^n$ with genus $\geq 1.$ The lowest eigenvalue of $L$ acting on scalar functions on $\Sigma$ has upper bound given by
$$ \eta_0 \leq -1 +\sup \limits_{x \in \Sigma, |v|=1} A^\nb(v, i) A^\nb(i,v).$$
\end{theorem}

{\bf Remark:} Note that if we combine Theorem 3.1 with Theorem 3.2 we find that our bound for $\eta_0$ is no better than $\eta_0 \leq -1/2$.

\begin{proof}

Let $\phi \in C^\infty_0(\Sigma)$ and $\omega$ be any non-zero GHF on $\Sigma$ with daul vector field $W.$ Also, let
$M_p = \sup \limits_{v \in T_p\Sigma, |v|=1} A^\nb(v, i) A^\nb(i,v) $. Note that $M_p$ depends on $p \in \Sigma$ and is not the supremum over $\Sigma$. Consider the tangent vector field $\phi W$. Plugging the coordinate functions of $\phi W$ into the definition of $\eta_0$ we get
$$ \eta_0 \intS \phi^2 |W|^2 \leq \intS(|\nabla^E (\phi W)|^2 - |A|^2 \phi^2 |W|^2 - \frac{1}{2}\phi^2|W|^2).$$

Note that our expression involves the Euclidean connection $\nabla^E$. As in the proof of Theorem 3.1, we have that
$$ \intS(|\nabla^E (\phi W)|^2 = \intS |\nabla \phi|^2 |W|^2 - \intS \phi^2 \la W, \mathcal L^E W \ra.$$

Now using equation (2.1) of Lemma 2.2 we have that $-\intS \phi^2 \la W, \mathcal L^E W \ra \leq \intS \phi^2 |W|^2(2M_p - 1/2).$ Then, using standard cut-off functions of increasing domain and $|\nabla^\Sigma \phi|^2 \leq 1$, we get that

$$ \eta_0 \intS |W|^2 \leq \intS |W|^2 (2M_p - |A|^2  - 1).$$

Using that $M_p - |A|^2(p) \leq 0$, we get the theorem.

\end{proof}


\section{\texorpdfstring{Applications to Co-Dimension One in $\mathbb R^3$}{Applications to Co-Dimension One in R3}}

Now, we only consider $\Sigma \to \mathbb R^3$. As in Fischer-Cobrie \cite{FC1985}, the index of the operator $L$ acting on scalar functions on $B_R \subset \Sigma$ is increasing in $R$ for any exhaustion of $\Sigma$ by $B_R$. The index of $L$ on $\Sigma$ is defined to be $\text{Index}_\Sigma(L) = \sup \limits_R (\text{Index}_{B_R}(L))$.

 Following the work of Ros\cite{Ros2006} and Urbano \cite{Urbano2011} on the Jacobi operator on minimal surfaces, we may give lower bounds for the index of $L$ if we have a condition on the principal curvatures. That is, we have

\begin{theorem}
Let $\Sigma$ be a 2-dimensional orientable self-shrinker of polynomial volume growth immersed in $\mathbb R^3$ with genus $g$ and principal curvatures $\kappa_i$. If \\
$|\kappa_1^2 - \kappa_2^2| \leq \delta < 1$, then the index of $L$ acting on scalar functions of $\Sigma$ has a lower bound given by

$$ \text{Index}_\Sigma(L) \geq \frac{g}{3}.$$
\end{theorem}

\begin{proof}
We may assume $\text{Index}_\Sigma (L) = J < \infty.$ As in Fischer-Colbrie \cite{FC1985}, there exist $\ltwo$ functions $\psi_1,...,\psi_J$ such that if  $f \in C^\infty_0(\Sigma)$ and $\intS f \psi_i = 0$ for all $i$ then $-\intS f Lf \geq 0$.

Similar to Farkas-Kra\cite{FK1980} p.42, we have $g$ linearly independent $\ltwo$ GHF's $\omega_i$ with dual vectors $W_i$. Allowing for $g=\infty,$ we define $V = \text{span\{ Finite number of }W_i\text{'s\}}$ where we are considering these to be vector fields with values in $\mathbb R^3$. Consider any $\phi \in C^\infty_0(\Sigma).$ Similar to the calculation in Theorem 3.1, we have that for any $W \in V$ that
$$ - \intS \la \phi W , L^E (\phi W) \ra = \intS |W|^2 |\nabla \phi|^2 - \intS \phi^2 \la W, L^E W \ra.$$

Using equation (2.1) of Lemma 2.2 we get that 
$$\intS \phi^2 \la W, L^E W \ra = \intS \phi^2 (|W|^2 + \sum \limits_{i\neq j} (\kappa_j^2 - \kappa_i^2)W_i^2)$$
$$ \geq \intS \phi^2 |W|^2 (1 - |\kappa_1^2 - \kappa_2^2|) \geq (1-\delta) \intS \phi^2 |W|^2.$$

Since $\dim V < \infty$, by using a standard cut-off function of large enough domain and $|\nabla \phi|^2$ small enough, we may ensure that $\dim V = \dim \phi V$ and that $- \intS \la \phi W, L^E (\phi W) \ra \leq 0$ for all $\phi W \in \phi V$.

We consider the linear map $F: \phi V \to \mathbb R^{3J}$ given by
$$ F(\phi W) = (\intS \phi W\psi_1,\,\, ...\,\, , \intS \phi W \psi_J).$$

For any $\phi W \in \phi V$, if $\phi W \in \text{Ker} F$ then each of its coordinate functions is orthogonal to every $\psi_i$. Therefore, $-\intS \la \phi W, L^E \phi W \ra \geq 0$. Therefore, since $\phi W \in \phi V$ we have that $\intS \la \phi W, L^E \phi W \ra = 0$. Hence, $\nabla^E (\phi W) \equiv 0$ and $\phi W$ is a constant vector. Since $\phi W$ has compact support, we must have that $\phi W \equiv 0$. Therefore, $\text{Ker} F = 0$.

So $\dim V \leq 3J$. By choosing larger and larger subspaces $V$ of $\text{span}\{W_i\}$, we have the theorem.
\end{proof}

For the case of compact $\Sigma$ with genus $\geq 1$ we may give another bound for the lowest eigenvalue of $L$. We may also give a bound for $\inf \limits_{x\in\Sigma}|x|^2$ if $|\kappa_1^2-\kappa_2^2|\leq \delta < 5/2.$

\begin{theorem}
Let $\Sigma$ be a 2-dimensional orientable compact self-shrinker immersed in $\mathbb R^3$ with genus $g \geq 1$ and principal curvatures $\kappa_i$. Let $\eta_0$ be the lowest eigenvalue of $L$ acting on scalar functions on $\Sigma$. We have that
$$ \eta_0 \leq -3/2 + \sup \limits_{x\in\Sigma} |\kappa_1^2 - \kappa_2^2| .$$

If $ |\kappa_1^2 - \kappa_2^2| \leq \delta < 5/2$, then
$$ \inf \limits_{x\in\Sigma} |x|^2 \leq \frac{4}{5/2 - \delta}.$$

\end{theorem}

\begin{proof}

Let $u$ be the eigenfunction for the the lowest eigenvalue $\eta_0$. Note that by standard theory, $u > 0$. Let $\omega$ be a GHF with $W$ its dual vector field. Consider the equation
$$ \intS |W|^2 Lu = -\eta_0 \intS |W|^2 u .$$

Perform integration by parts on the LHS, use that \\
$L|W|^2 \geq 2\la W, \mathcal L^\Sigma W \ra + (|A|^2 + \frac{1}{2})|W|^2,$ and use equation (1.1) of Corollary 1.1 to get
$$ -\eta_0 \intS |W|^2 u \geq \intS \frac{3}{2}|W|^2 u + \intS u \sum \limits_{i \neq j} (\kappa_i^2 - \kappa_j^2)W_j^2.$$

Since $u |W|^2 \geq 0$ and $u|W|^2 \neq 0$ we get that for some $i,j$ and point $p \in \Sigma$ that $0 \geq \frac{3}{2}+\eta_0 + \kappa_i^2-\kappa_j^2$. So we get the first inequality of the theorem:
$$ \eta_0 \leq \kappa_j^2 - \kappa_i^2 - 3/2 \leq -3/2 + \sup |\kappa_1^2 - \kappa_2^2|.$$

Now, consider the case that $|\kappa_1^2 - \kappa_2^2| \leq \delta < 5/2$. Since we are on a two-dimensional self-shrinker, we know from Colding-Minicozzi \cite{CM2009} that $L |x|^2 = 4 - |x|^2$. We use this equality to get
\begin{equation} \intS |W|^2 L|x|^2 = 4\intS |W|^2 - \intS |W|^2|x|^2. \end{equation}

We perform integration by parts on the LHS to get
\begin{equation} \intS |W|^2 L|x|^2 = \intS |x|^2 L|W|^2 \geq \intS 3/2|x|^2|W|^2 + \intS |x|^2 \sum \limits_{i\neq j} (\kappa_i^2 - \kappa_j^2) W_j^2.\end{equation}

Putting equations (4.1) and (4.2) together, we get
$$ 4\intS |W|^2 \geq (5/2 - \delta) \intS |x|^2 |W|^2 \geq (5/2 - \delta)(\inf |x|^2)\intS |W|^2.$$
Since $W \neq 0,$ we get the second inequality of the theorem.

\end{proof}


\bibliographystyle{plain}
\bibliography{GHFref}

\end{document}